\documentclass[12pt]{amsart}
\usepackage{graphicx}
\usepackage[active]{srcltx}
\usepackage{amsthm,mathrsfs}
\usepackage{a4wide}
\usepackage{amsfonts}
\usepackage{amsmath}
\usepackage{amssymb}
\newtheorem{lemma}{Lemma}
\newtheorem{theorem}{Theorem}

\usepackage{float}
\usepackage{url}
\usepackage{enumitem}

\makeatletter\def\blfootnote{\xdef\@thefnmark{}\@footnotetext}\makeatother
\allowdisplaybreaks
\title{\bf On sequences with prescribed metric discrepancy behavior}

\author{Christoph Aistleitner} 
\address{ Institute of Financial Mathematics and applied Number Theory, University Linz}
\email{christoph.aistleitner@jku.at}

\author{Gerhard Larcher} 
\address{ Institute of Financial Mathematics and applied Number Theory, University Linz}
\email{gerhard.larcher@jku.at}

\thanks{The first author is supported by a Schr\"odinger scholarship of the Austrian Science Fund (FWF). The second author is supported by the Austrian Science Fund (FWF): Project F5507-N26, which is part of the Special Research Program ``Quasi-Monte Carlo Methods: Theory and Applications''}

\begin{document}

\begin{abstract}
An important result of H. Weyl states that for every sequence $\left(a_{n}\right)_{n\geq 1}$ of distinct positive integers the sequence of fractional parts of $\left(a_{n} \alpha \right)_{n \geq1}$ is uniformly distributed modulo one for almost all $\alpha$. However, in general it is a very hard problem to calculate the precise order of convergence of the discrepancy $D_{N}$ of $\left(\left\{a_{n} \alpha \right\}\right)_{n \geq 1}$ for almost all $\alpha$. By a result of R. C. Baker this discrepancy always satisfies $N D_{N} = \mathcal{O} \left(N^{\frac{1}{2}+\varepsilon}\right)$ for almost all $\alpha$ and all $\varepsilon >0$. In the present note for arbitrary $\gamma \in \left(0, \frac{1}{2}\right]$ we construct a sequence $\left(a_{n}\right)_{n \geq 1}$ such that for almost all $\alpha$ we have $ND_{N} = \mathcal{O} \left(N^{\gamma}\right)$ and $ND_{N} = \Omega \left(N^{\gamma-\varepsilon}\right)$ for all $\varepsilon > 0$, thereby proving that any prescribed metric discrepancy behavior within the admissible range can actually be realized.
\end{abstract}

\date{}
\maketitle

\section{Introduction} \label{sect_1}
H. Weyl~\cite{wey} proved that for every sequence $\left(a_{n}\right)_{n \geq 1}$ of distinct positive integers the sequence $\left(\left\{a_{n} \alpha\right\}\right)_{n \geq 1}$ is uniformly distributed modulo one for almost all reals $\alpha$. Here, and in the sequel, $\{ \cdot \}$ denotes the fractional part function. The speed of convergence towards the uniform distribution is measured in terms of the discrepancy, which -- for an arbitrary sequence $\left(x_{n}\right)_{n \geq 1}$ of points in $\left[ \left. 0,1\right.\right)$ -- is defined by
$$
D_N = D_{N} (x_1, \dots, x_N) = \underset{0 \leq a < b \leq 1}{\sup} \left|\frac{\mathcal{A}_{N}\left(\left[\left.a,b\right)\right.\right)}{N} - \left(b-a\right)\right|,
$$
where $\mathcal{A}_{N}\left(\left[\left.a,b\right)\right.\right) := \# \left\{1 \leq n \leq N \left| \right. x_{n} \in \left[\left. a,b\right.\right)\right\}.$  For a given sequence $\left(a_{n}\right)_{n \geq 1}$ it is usually a very hard and challenging problem to give sharp estimates for the discrepancy $D_{N}$ of $\left(\left\{a_{n} \alpha\right\}\right)_{n \geq 1}$ valid for almost all $\alpha$. For general background on uniform distribution theory and discrepancy theory see for example the monographs~\cite{dts,knu}.\\

A famous result of R. C. Baker~\cite{baker} states that for any sequence $\left(a_{n}\right)_{n \geq 1}$ of distinct positive integers for the discrepancy $D_{N}$ of $\left(\left\{a_{n} \alpha\right\}\right)_{n \geq 1}$ we have 
\begin{equation} \label{B*}
ND_{N}=\mathcal{O} \left(N^{\frac{1}{2}} \left(\log N\right)^{\frac{3}{2}+\varepsilon}\right) \qquad \textrm{as $N \to \infty$}
\end{equation}
for almost all $\alpha$ and for all $\varepsilon >0$.\\

Note that~\eqref{B*} is a general upper bound which holds for \emph{all} sequences $(a_n)_{n \geq 1}$; however, for some specific sequences the precise typical order of decay of the discrepancy of $(\{a_n \alpha\})_{n \geq 1}$ can differ significantly from the upper bound in~\eqref{B*}. The fact that~\eqref{B*} is essentially optimal (apart from logarithmic factors) as a general result covering all possible sequences can for example be seen by considering so-called lacunary sequences $\left(a_{n}\right)_{n \geq 1}$, i.e., sequences for which $\frac{a_{n+1}}{a_{n}} \geq 1+\delta$ for a fixed $\delta > 0$ and all $n$ large~enough. In this case for $D_{N}$ we have
$$
\frac{1}{4 \sqrt{2}} \leq \underset{N \rightarrow \infty}{\lim \sup} \frac{ND_{N}}{\sqrt{2 N \log \log N}} \leq c_{\delta}
$$
for almost all $\alpha$ (see~\cite{philipp}), which shows that the exponent $1/2$ of $N$ on the right-hand side of~\eqref{B*} cannot be reduced for this type of sequence. For more information concerning possible improvements of the logarithmic factor in~\eqref{B*}, see~\cite{bpt}.\\

Quite recently in~\cite{aist-lar} it was shown that also for a large class of sequences with polynomial growth behavior Baker's result is essentially best possible. For example, the following result was shown there: Let $f \in \mathbb{Z} \left[x\right]$ be a polynomial of degree larger or equal to 2. Then for the discrepancy $D_{N}$ of $\left(\left\{f (n) \alpha\right\}\right)_{n \geq 1}$ for almost all $\alpha$ and for all $\varepsilon > 0$ we have 
$$
N D_{N} = \Omega \left( N^{\frac{1}{2}-\varepsilon} \right).
$$

On the other hand there is the classical example of the Kronecker sequence, i.e., $a_{n} = n$, which shows that the actual metric discrepancy behavior of $(\{a_n \alpha\})_{n \geq 1}$ can differ vastly from the general upper bound in~\eqref{B*}. Namely, for the discrepancy of the sequence $\left(\left\{n \alpha\right\}\right)_{n \geq 1}$ for almost all $\alpha$ and for all $\varepsilon > 0$ we have
\begin{equation} \label{K*}
ND_{N} = \mathcal{O}\left(\log N \left(\log \log N\right)^{1+\varepsilon}\right),
\end{equation}
which follows from classical results of Khintchine in the metric theory of continued fractions (for even more precise results, see~\cite{schoi}). The estimate~\eqref{K*} of course also holds for $a_{n} = f(n)$ with $f \in \mathbb{Z} \left[x\right]$ of degree 1. In~\cite{aist-lar} further examples for $\left(a_{n}\right)_{n \geq 1}$ were given, where $\left(a_{n}\right)_{n \geq 1}$ has polynomial growth behavior of arbitrary degree, such that for the discrepancy of $\left(\left\{a_{n} \alpha \right\}\right)_{n \geq 1}$ we have 
$$
ND_{N} = \mathcal{O}\left(\left(\log N\right)^{2+\varepsilon}\right)
$$
for almost all $\alpha$ and for all $\varepsilon >0$; see there for more details.\\

These results may seduce to the hypothesis that for all choices of $\left(a_{n}\right)_{n \geq 1}$ for the discrepancy of $\left(\left\{a_{n} \alpha \right\}\right)_{n \geq 1}$ for almost all $\alpha$ we either have
\begin{equation} \label{equ_a}
ND_{N} = \mathcal{O} \left(N^{\varepsilon}\right)
\end{equation}
or
\begin{equation} \label{equ_b}
ND_{N} = \Omega\left(N^{\frac{1}{2}-\varepsilon}\right).
\end{equation}
This hypothesis, however, is wrong as was shown in~\cite{AHL}: Let $\left(a_{n}\right)_{n \geq 1}$ be the sequence of those positive integers with an even sum of digits in base 2, sorted in increasing order; that is $(a_n)_{n \geq 1} = (3,5,6,9,10,\dots)$. Then for the discrepancy of $\left(\left\{a_{n} \alpha\right\}\right)_{n \geq 1}$for almost all $\alpha$ we have 
$$
ND_{N} = \mathcal{O}\left(N^{\kappa+\varepsilon}\right)
$$
and
$$ND_{N} = \Omega\left(N^{\kappa- \varepsilon}\right)
$$
for all $\varepsilon >0$, where $\kappa$ is a constant with $\kappa \approx 0,404$. Interestingly, the precise value of $\kappa$ is unknown; see~\cite{fru} for the background.\\

The aim of the present paper is to show that the example above is not a singular counter-example, but that indeed ``everything'' between~\eqref{equ_a} and~\eqref{equ_b} is possible. More precisely, we will show the following theorem.

\begin{theorem} \label{th_a}
Let $0 < \gamma \leq \frac{1}{2}$. Then there exists a strictly increasing sequence $\left(a_{n}\right)_{n \geq 1}$ of positive integers such that for the discrepancy of the sequence $\left(\left\{a_{n} \alpha \right\}\right)_{n \geq 1}$ for almost all $\alpha$ we have 
$$
ND_{N} = \mathcal{O}\left(N^{\gamma}\right)
$$
and
$$
ND_{N} = \Omega\left(N^{\gamma-\varepsilon}\right)
$$
for all $\varepsilon >0$.
\end{theorem}

\section{Proof of the Theorem} \label{th_b}

For the proof we need an auxiliary result which easily follows from classical work of H. Behnke~\cite{beh}. 

\begin{lemma} \label{lem_a}
Let $\left(e_{k}\right)_{k \geq 1}$ be a strictly increasing sequence of positive integers. Let $\varepsilon > 0$. Then for almost all $\alpha$ there is a constant $K\left(\alpha, \varepsilon \right) > 0$ such that for all $r \in \mathbb{N}$ there exist $M_{r} \leq e_{r}$ such that for the discrepancy of the sequence $\left(\left\{n^{2} \alpha\right\}\right)_{n \geq 1}$ we have
$$
M_{r} D_{M_{r}} \geq K\left(\alpha, \varepsilon\right) \sqrt{\frac{e_{r}}{\left(\log e_{r}\right)^{1+\varepsilon}}}.
$$
\end{lemma}

\begin{proof}
For $\alpha \in \mathbb{R}$ let $a_{k} \left(\alpha\right)$ denote the $k$-th continued fraction coefficient in the continued fraction expansion of $\alpha$. Then it is well-known that for almost all $\alpha$ we have $a_{k}(\alpha) = \mathcal{O} \left(k^{1+\varepsilon}\right)$ for all $\varepsilon >0$. Let $\varepsilon >0$ be given and let $\alpha$ and $c\left(\alpha, \varepsilon\right)$ be such that 
\begin{equation} \label{equ_c}
a_{k} (\alpha) \leq c\left(\alpha, \varepsilon\right) k^{1+\varepsilon}
\end{equation}
 for all $k \geq 1$.\\

Let $q_{l}$ the $l$-th best approximation denominator of $\alpha$. Then
\begin{equation} \label{**}
q_{l+1} \leq \left(c\left(\alpha, \varepsilon\right) l^{1+\varepsilon}+1\right) q_{l}.
\end{equation}
Since $q_{l} \geq 2^{\frac{l}{2}}$ in any case, we have $l \leq \frac{2\log q_{l}}{\log 2}$, and we obtain 
\begin{equation} \label{***}
q_{l+1} \leq c_{1} \left(\alpha, \varepsilon\right) q_{l} \left(\log q_{l}\right)^{1+\varepsilon},
\end{equation}
for an appropriate constant $c_{1} \left(\alpha, \varepsilon\right)$. In~\cite{beh} it was shown in Satz XVII that for every real $\alpha$ we have
$$
\left| \sum^{N}_{n=1} e^{2 \pi i n^{2} \alpha} \right| = \Omega\left(N^{\frac{1}{2}}\right).
$$
Indeed, if we follow the proof of this theorem we find that even the following was shown: For every $\alpha$ and for every best approximation denominator $q_{l}$ of $\alpha$ there exists an $Y_{l} < \sqrt{q_{l}}$ such that $\left| \sum^{Y_{l}}_{n=1} e^{2 \pi i n^{2}\alpha} \right| \geq c_{\textup{abs}} \sqrt{q_{l}}$. Here $c_{\textup{abs}}$ is a positive absolute constant (not depending on~$\alpha$).\\

Let now $r \in \mathbb{N}$ be given and let $l$ be such that $q_{l} \leq e_{r} < q_{l+1}$, and let $M_{r} := Y_{l}$ from above. Then by~\eqref{**} and~\eqref{***} we obtain, for an appropriate constant $c_{2}\left(\alpha, \varepsilon\right)$,
\begin{eqnarray*}
\left| \sum^{M_{r}}_{n=1} e^{2 \pi i n^{2} \alpha} \right| & \geq & c_{\textup{abs}} \sqrt{q_{l}}\\
& \geq & c_{2} \left(\alpha, \varepsilon\right) \sqrt{\frac{q_{l+1}}{\left(\log q_{l}\right)^{1+\varepsilon}}}\\
& \geq & c_{2}\left(\alpha, \varepsilon\right) \sqrt{\frac{e_{l}}{\left(\log e_{l}\right)^{1+\varepsilon}}}.
\end{eqnarray*}
By the fact that (see Chapter 2, Corollary 5.1 of~\cite{knu})
$$
M_{r} D_{M_{r}} \geq \frac{1}{4} \left| \sum^{M_{r}}_{n=1} e^{2 \pi i n^{2} \alpha} \right|,
$$
which is a special case of Koksma's inequality, the result follows.
\end{proof}

Now we are ready to prove the main theorem.

\begin{proof}[Proof of the Theorem]
Let $\left(m_{j}\right)_{j \geq 1}$ and $\left(e_{j}\right)_{j \geq 1}$ be two strictly increasing sequences of positive integers, which will be determined later. We will consider the following strictly increasing sequence of positive integers, which will be our sequence $\left(a_{n}\right)_{n \geq 1}$:
\begin{eqnarray*}
&& 1,~2,~3,~ \ldots,~ \underbrace{m_{1}}_{ =: A_{1}},\\
&& m_{1}+1^{2},~ m_{1} +2^{2},~ m_{1}+3^{2},~ m_{1}+4^{2},~ \ldots,~ \underbrace{m_{1} + {e_{1}}^{2}}_{ := B_{1}},\\
&& B_{1}+1,~ B_{1}+2,~ B_{1}+3,~ \ldots,~ \underbrace{B_{1}+m_{2}}_{ =: A_{2}},\\
&& A_{2} +1^{2},~ A_{2}+2^{2},~ A_{2}+3^{2},~ A_{2}+4^{2},~ \ldots,~ \underbrace{A_{2} + {e_{2}}^{2}}_{ =: B_{2}},\\
&& B_{2}+1,~ B_{2}+2,~ B_{2}+3,~ \ldots,~ \underbrace{B_{2}+m_{3}}_{ =: A_{3}},\\
&& A_{3} +1^{2},~ A_{3}+2^{2},~ A_{3}+3^{2},~ A_{3}+4^{2},~ \ldots,~ \underbrace{A_{3} + {e_{3}}^{2}}_{ =: B_{3}},\\
&& \vdots \\
\end{eqnarray*}
Furthermore, let 
$$
F_{s}:= \sum^{s}_{i=1} m_{i} + \sum^{s-1}_{i=1} e_{i} \qquad \textrm{and} \qquad E_{s} := \sum^{s}_{i=1} m_{i} + \sum^{s}_{i=1} e_{i}.
$$
The sequence $(a_n)_{n \geq 1}$ is constructed in such a way that it contains sections where it grows like $(n)_{n \geq 1}$ as well as sections where it grows like $(n^2)_{n \geq 1}$. By this construction we exploit both the strong upper bounds for the discrepancy of $(\{n \alpha\})_{n \geq 1}$ and the strong lower bounds for the discrepancy of $(\{n^2 \alpha\})_{n \geq 1}$, in an appropriately balanced way, in order to obtain the desired discrepancy behavior of the sequence $(\{a_n \alpha\})_{n \geq 1}$. In our argument we will repeatedly make use of the fact that
\begin{equation} \label{E*}
D_N (x_1, \dots, x_N) = D_N(\{x_1 + \beta\}, \dots, \{x_N + \beta\})
\end{equation}
for arbitrary $x_1, \dots, x_N \in [0,1]$ and $\beta \in \mathbb{R}$, which allows us to transfer the discrepancy bounds for $(\{n \alpha\})_{n \geq 1}$ and $(\{n^2 \alpha\})_{n \geq 1}$ directly to the shifted sequences $(\{(M+n) \alpha\})_{n \geq 1}$ and $(\{(M+n^2) \alpha\})_{n \geq 1}$ for some integer $M$.\\

Let $\alpha$ be such that it satisfies~\eqref{equ_c} with $\varepsilon=\frac{1}{2}$. Then it is also well-known (see for example~\cite{knu}) that for the discrepancy $D_{N}$ of the sequence $\left(\left\{n \alpha\right\}\right)_{n \geq 1}$ we have 
\begin{equation} \label{equ_d}
ND_{N} \leq \overline{c}_{1} \left(\alpha \right) \left(\log N\right)^{\frac{3}{2}}
\end{equation}
for all $N \geq 2$.\\

By the above mentioned general result of Baker, that is by~\eqref{B*}, we know that for almost all $\alpha$ for the discrepancy $D_{N}$ of the sequence $\left(\left\{n^{2} \alpha \right\}\right)_{n\geq 1}$ we have
$$
ND_{N} \leq c_{3} \left(\alpha, \varepsilon \right) N^{\frac{1}{2}} \left(\log N\right)^{\frac{3}{2}+\varepsilon}
$$
for all $\varepsilon > 0$ and for all $N \geq 2$, for an appropriate constant $c_{3} \left(\alpha, \varepsilon \right)$. Actually an even slightly sharper estimate was given for the special case of the sequence $(\{n^2 \alpha\})_{n \geq 1}$ by Fiedler, Jurkat and K\"orner in~\cite{FJK}, who proved that
\begin{equation} \label{equ_e}
ND_{N} \leq c_{4} \left(\alpha, \varepsilon\right) N^{\frac{1}{2}} \left(\log N\right)^{\frac{1}{4}+\varepsilon}
\end{equation}
for almost all $\alpha$ and for all $\varepsilon > 0$ and all $N \geq 2$.\\

Assume that $\alpha$ satisfies~\eqref{equ_e} with $\varepsilon = \frac{1}{8}$. Then
\begin{equation} \label{equ_f}
ND_{N} \leq \overline{c}_{2} \left(\alpha\right) N^{\frac{1}{2}} \left(\log N\right)^{\frac{3}{8}}
\end{equation}
for all $N \geq 2$. Now for such $\alpha$ and for arbitrary $N$ we consider the discrepancy $D_{N}$ of the sequence $\left(\left\{a_{n} \alpha\right\}\right)_{n \geq 1}$.\\

\emph{Case~1.}\\
Let $N=F_{l}$ for some $l$. Then $ND_{N} \leq E_{l-1} D_{E_{l-1}}+\left(N-E_{l-1}\right) D_{E_{l-1},F_{l}},$ where $D_{x,y}$ denotes the discrepancy of the point set $\left(\left\{a_{n} \alpha \right\}\right)_{n=x+1, x+2, \ldots, y}$. Hence by~\eqref{E*},~\eqref{equ_d} and by the trivial estimate $D_{B_{l-1}} \leq 1$ we have
\begin{eqnarray*}
ND_{N} & \leq & E_{l-1} + \overline{c}_{1} \left(\alpha\right) \left(\log m_{l}\right)^{\frac{3}{2}}\\
& \leq & 2 \left(\log m_{l}\right)^{2}\\
& \leq & 2 \left(\log N\right)^{2}
\end{eqnarray*}
for all $l$ large enough, provided that (condition (i)) $m_{l}$ is chosen such that $\left(\log m_{l}\right)^{2} \geq~E_{l-1}$.\\

\emph{Case~2.}\\
Let $F_{l} < N \leq E_{l}$ for some $l$. Then by Case~1 and by~\eqref{E*} and~\eqref{equ_f} we have for $l$ large enough that
\begin{eqnarray*}
ND_{N} & \leq & F_{l} D_{F_{l}} + \left(N-F_{l}\right)D_{F_{l},N}\\
& \leq & 2 \left(\log F_{l}\right)^{2} + \overline{c}_{2} \left(\alpha\right) \left(N-F_{l}\right)^{\frac{1}{2}} \left(\log \left(N-F_{l}\right)\right)^{\frac{3}{8}}.
\end{eqnarray*}
Note that $0<N-F_{l} < e_{l}$.\\

We choose (condition (ii))
\begin{equation} \label{equ_g}
e_{l} := \left\lceil \frac{{F_{l}}^{2 \gamma}}{\log \left({F_{l}}^{2 \gamma}\right)}\right\rceil.
\end{equation}
Note that conditions (i) and (ii) do not depend on $\alpha$. Now assume that $l$ is so large that $2 \left(\log F_{l}\right)^{2} < \frac{{F_{l}}^{\gamma}}{2}$. Then
$$
\frac{{F_{l}^{\gamma}}}{2} \leq 2 \left(\log F_{l}\right)^{2} + \left(e_{l} \log e_{l}\right)^\frac{1}{2} \leq 2 {F_{l}^{\gamma}}
$$
and (note that $\gamma \leq \frac{1}{2}$)
\begin{equation} \label{equ_h}
F_{l} < N \leq E_{l} = F_{l} + e_{l} \leq 2 F_{l}.
\end{equation}
Hence
\begin{eqnarray*}
ND_{N} & \leq & \max\left(1, \overline{c}_{2} \left(\alpha\right)\right) 2 {F_{l}^{\gamma}}\\
& \leq & \max\left(1, \overline{c}_{2} \left(\alpha\right)\right) 2^{1+\gamma}  N^{\gamma}.
\end{eqnarray*}

\emph{Case~3.}\\
Let $E_{l} < N < F_{l+1}$ for some $l$. Then by Case 2 and by~\eqref{E*} and~\eqref{equ_d} we have
\begin{eqnarray*}
ND_{N} & \leq & E_{l} D_{E_{l}} + \left(N-E_{l}\right) D_{E_{l}, N}\\
& \leq & {E_{l}^{\gamma}} + \overline{c}_{1} \left(\alpha\right) \left(\log \left(N-E_{l}\right)\right)^{2}\\
& \leq & 2N^{\gamma}
\end{eqnarray*}
for $N$ large enough.\\

It remains to show that for every $\varepsilon > 0$ we have $ND_{N} \geq N^{\gamma-\varepsilon}$ for infinitely many $N$. Let $l$ be given and let $M_{l} \leq e_{l}$ with the properties given in Lemma~\ref{lem_a}. Let $N:= F_{l}+M_{l}$. Then by Lemma~\ref{lem_a}, Case~1,~\eqref{E*},~\eqref{equ_g} and~\eqref{equ_h} for $l$ large enough we have
\begin{eqnarray*}
ND_{N} & \geq & M_{l} D_{F_{l}, N} -F_{l} D_{F_{l}}\\
& \geq & K\left(\alpha, \varepsilon\right) \sqrt{\frac{e_{l}}{\left(\log e_{l}\right)^{1+\varepsilon}}}-2 \left(\log m_{l}\right)^{2}\\
& \geq & \frac{{F_{l}^{\gamma}}}{\left(\log F_{l}\right)^{3}} \\
& \geq & N^{\gamma-\varepsilon}.
\end{eqnarray*}
This proves the theorem.
\end{proof}

\end{document}